\documentclass[11pt]{amsart}
\usepackage{amssymb}
\usepackage{amscd}
\usepackage[all]{xy}
\usepackage{comment}

\numberwithin{equation}{section}

\def\today{\number\day\space\ifcase\month\or   January\or February\or
March\or April\or May\or June\or   July\or August\or September\or
October\or November\or December\fi\   \number\year}

\newtheorem{Thm}{Theorem}
\newtheorem{Cor}[Thm]{Corollary}

\newtheorem*{Claim-nonum}{Claim}

\theoremstyle{definition}

\newtheorem{qst}[Thm]{Question}
\newtheorem{rmk}[Thm]{Remark}

\newcounter{CnsEnumi}

\newcommand{\beq}{\begin{equation}}
\newcommand{\eeq}{\end{equation}}
\newcommand{\beqr}{\begin{eqnarray*}}
\newcommand{\eeqr}{\end{eqnarray*}}
\newcommand{\bal}{\begin{align*}}
\newcommand{\eal}{\end{align*}}
\newcommand{\bei}{\begin{itemize}}
\newcommand{\eei}{\end{itemize}}

\newcommand{\gm}{\gamma}

\newcommand{\ep}{\varepsilon}

\newcommand{\kp}{\kappa}

\newcommand{\rh}{\rho}

\newcommand{\Gm}{\Gamma}

\newcommand{\Z}{{\mathbb{Z}}}
\newcommand{\R}{{\mathbb{R}}}
\newcommand{\C}{{\mathbb{C}}}
 \newcommand{\N}{{\mathbb{N}}}

\newcommand{\T}{\mathbb{T}}

\newcommand{\id}{{\operatorname{id}}}

\newcommand{\diag}{{\operatorname{diag}}}

\newcommand{\Aut}{{\operatorname{Aut}}}

\newcommand{\Ell}{{\operatorname{Ell}}}

\newcommand{\Zh}{\mathcal{Z}}

\newcommand{\rc}{\mathrm{rc}}

\newcommand{\eps}{\varepsilon}
\numberwithin{equation}{section}

\newcommand{\tgamma}{\Gamma}

\newcommand{\dirlim}{\varinjlim}

\newcommand{\andeqn}{\qquad {\mbox{and}} \qquad}

\newcommand{\hm}{homomorphism}

\newcommand{\pj}{projection}

\newcommand{\cfn}{continuous function}

\title{Exotic circle actions on classifiable $C^*$-algebras}

\author{Ilan Hirshberg}
\address{Department of Mathematics, Ben Gurion University of the Negev,
P.O.B. 653, Be'er Sheva 84105, Israel}

\thanks{This work was supported by the Israel Science Foundation grant 
no.~476/16. I thank the Fields Institute, where some of this work on this paper was done, for its hospitality} 

\begin{document}

\begin{abstract}
We construct an example of a simple nuclear separable unital stably finite 
$\Zh$-stable 
$C^*$-algebra 
along with an action of the circle such that the crossed product is simple but 
not $\Zh$-stable. 
\end{abstract}

\maketitle

Classification theory for simple nuclear $C^*$-algebras
reached a milestone in the last decade.
The results of \cite{elliott-gong-lin-niu}
and \cite{tikuisis-white-winter},
building on decades of work by many authors,
show that $\Zh$-stable simple separable nuclear unital $C^*$-algebras
satisfying the Universal Coefficient Theorem
are classified via the Elliott invariant,
$\Ell (\cdot)$, which consists of the ordered $K_0$-group
along with the class of the identity, the $K_1$-group,
the trace simplex,
and the pairing
between the trace simplex and the $K_0$-group.
Earlier counterexamples
due to Toms and R{\o}rdam
(\cite{toms-counterexample, rordam-counterexample}),
related to ideas of Villadsen (\cite{Villadsen-perforation}),
show that one cannot expect to be able to extend
this classification theorem beyond the $\Zh$-stable case,
at least without either extending the invariant
or restricting to another class of $C^*$-algebras.

The property of $\Zh$-stability, known to be equivalent to finite nuclear dimension and conjectured to be equivalent to strict comparison for simple separable nuclear infinite dimensional $C^*$-algebras, is thus a key regularity property. A central question, then, is to find conditions which ensure that various natural constructions imply $\Zh$-stability. In particular, one would want to know how robust is the class of classifiable $C^*$-algebras under various natural constructions, particularly crossed products by actions of amenable groups. 

For actions of countable discrete amenable groups, it is known that if $\alpha$ 
is a strongly outer action of such a group on a simple separable unital nuclear 
$\Zh$-stable $C^*$-algebra $A$, then the crossed product is again $\Zh$-stable 
under each of the following conditions:
\begin{enumerate}
	\item The trace space $T(A)$ is a Bauer simplex with finite dimensional 
	extreme boundary, the action of $G$ on the extreme boundary factors through 
	a finite group action. See \cite{GHV}, and also 
	\cite{matui-sato,sato} for earlier similar results with stronger conditions 
	on the action.
	\item The trace space $T(A)$ is a Bauer simplex with finite dimensional extreme boundary, $G = \Z$ with no restrictions on the action. See \cite{wouters}.
	\item The trace space $T(A)$ is a Bauer simplex with finite dimensional 
	extreme boundary and the action on $\partial_e (T(A))$ is free. See \cite{GGNV}, and 
	related results with a stronger conclusion but with stronger hypotheses in
	\cite{wouters}.
	\item The trace space $T(A)$ is a Bauer simplex (with possibly infinite 
	dimensional extreme boundary), $G = \Z^d$ , and the action of $G$ on  
	$\partial_e 
	(T(A))$ is free and minimal. See \cite{GGNV}.
\end{enumerate}
The first three results above in fact show somewhat more: they show that the 
action is cocycle-conjugate to an action which tensorially absorbs the trivial 
action on $\Zh$. They also show that such actions have finite Rokhlin 
dimension. We refer the reader to \cite{HWZ,HSWW,Gadella-compact,SWZ} for the 
definitions and discussion of Rokhlin dimension for actions of finite groups 
and $\Z$, for actions of $\R$, for actions of compact groups, and for actions 
of residually finite groups, respectively.

The purpose of this paper is to provide a negative result when the group is not 
discrete: we provide an example of an action of the circle group on a simple 
separable unital $\Zh$-stable $C^*$-algebra such that the crossed product is 
simple but not $\Zh$-stable. That is stated as Corollary \ref{cor:main}. 

It is easy to construct such an example if we do not assume that the crossed 
product is simple. Indeed, if $A$ is obtained as a crossed product $A \cong 
C(X) \rtimes_{\alpha} \Z$, then the crossed product by the dual action 
satisfies, by Takai duality, $A \rtimes_{\hat{\alpha}} \T \cong C(X) \otimes 
K$, which is certainly not $\Zh$-stable. In fact, if $X$ is chosen to be 
infinite dimensional and the action $\alpha$ has zero mean dimension, then $A$ 
is $\Zh$-stable by the main theorem in \cite{Elliott-Niu}, but $A 
\rtimes_{\hat{\alpha}} \T$ 
doesn't even have finite nuclear dimension. This is not surprising. However, 
one might have been tempted to conjecture that a suitable outerness condition 
would suffice. The natural outerness condition here would be full strong Connes 
spectrum (see 
\cite{Kishimoto-strong-connes}), which for actions on simple $C^*$-algebras is 
equivalent to requiring that the crossed product is simple. Taking cue from the 
finite group case, one might have even hoped that this would imply finite 
Rokhlin dimension. The example in this paper shows that the situation is more 
complicated.

In this paper, if $A$ is a unital $C^*$-algebra, by a (normalized) trace on $A$ we mean a tracial state. A tracial state $\tau$ induces a non-normalized trace on $M_{\infty}(A)_+$ given by $\tau(a) = \sum_{k=1}^{\infty} a_{kk}$, where we think of $a = (a_{kl})_{k,l = 1,2,3\ldots} $ as an infinite matrix with finitely many non-zero coefficients in $A$ (and slightly abusing notation to use $\tau$ for the extension as well). For $a \in M_{\infty}(A)_+$, we denote $d_{\tau}(a) = \lim_{n \to \infty} \tau (a^{1/n})$. For $r \in [0,\infty)$ we say that $A$ has $r$-comparison if for any $a,b \in M_{\infty}(A)_+$, if $d_{\tau} (a) + r < d_{\tau}(b)$ for any normalized quasitrace $\tau$ on $A$ then $a$ is Cuntz subequivalent to $b$. As we are dealing with nuclear $C^*$-algebras in this paper, any quasitrace is a trace (\cite{haagerup}). The radius of comparison of $A$ is defined to be the infimum the $r$'s for which $A$ has $r$-comparison. We refer the reader to \cite{Toms-flat-dim-growth} for 
a discussion of the radius of comparison.

I thank M.~Ali Asadi-Vasfi, Apurva Seth and the referee for their comments on an earlier draft of this paper.

\begin{Thm}\label{thm:main}
	There exists a simple unital separable AH-algebra $C$ which is not 
	$\Zh$-stable and admits a properly outer automorphism $\alpha$ such that 
	$A \rtimes_{\alpha} \Z$ is $\Zh$-stable.
\end{Thm}
\begin{Cor}
	\label{cor:main}
	There exists a simple unital separable nuclear $\Zh$-stable $C^*$-algebra 
	$B$ in the UCT class along with a point-norm continuous action $\beta 
	\colon \T \to \Aut(B)$ such that $B \rtimes_{\beta} \T$ is simple but not 
	$\Zh$-stable.
\end{Cor}
\begin{proof}[Proof of Corollary \ref{cor:main}] Let $C$ and $\alpha$ be as in 
the statement of Theorem~\ref{thm:main}. Set $B = C \rtimes_{\alpha} \Z$. As 
the UCT class is closed under crossed products by $\Z$ 
(\cite[22.3.5(g)]{blackadar_book}), $B$ is in the UCT class. Let 
$\beta = \hat{\alpha}$. Then $B$ is $\Zh$-stable but $B \rtimes_{\beta} \T 
\cong C \otimes K$ is simple but not $\Zh$-stable.
\end{proof}

\begin{proof}[Proof of Theorem \ref{thm:main}]
	We first explain the idea of the proof. We start with a Villadsen-type 
	construction, of the form
\[
		\begin{xy}
		(0,0)*+{C(X_0) }="base";
		(36,0)*+{ C(X_1) \otimes M_{n_1} }="0";
		(72,0)*+{ C(X_2) \otimes M_{n_2} }="00";
		(108,0)*+{ C(X_3) \otimes M_{n_3} }="000";
		{\ar@2 "base";"0"};
		{\ar@2 "0";"00"};
		{\ar@2 "00";"000"};
	\end{xy}
\]
where the spaces are products of spheres, and the connecting maps consist of 
various coordinate projections and a small number of point evaluations intended 
to make the inductive limit simple. We now take this inductive system and 
overlay it on top of a binary tree: we replace the $n$'th element in the 
sequence by a direct sum of $2^n$ copies of it, and copy each outgoing 
homomorphism twice so it ends up in two copies, as follows.
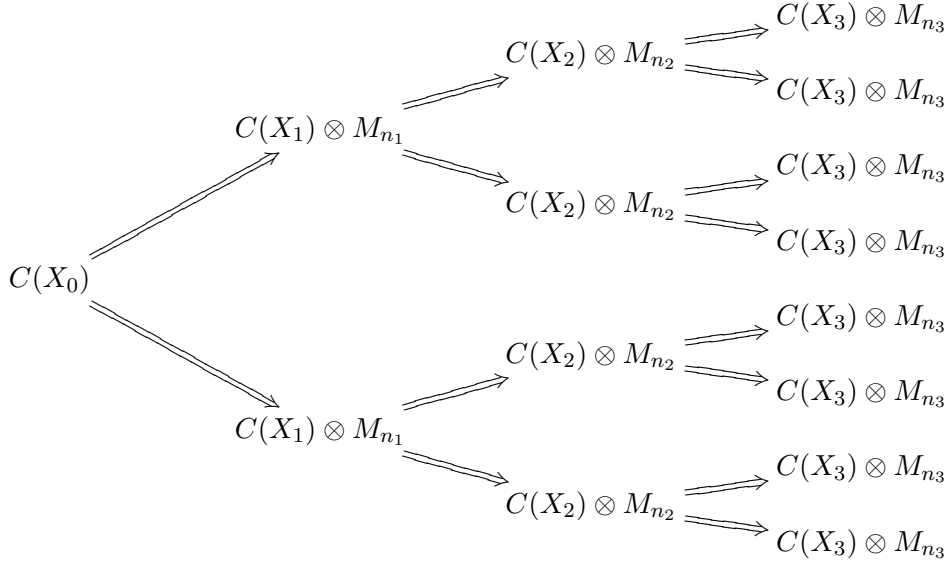
\[
\begin{xy}
	(0,0)*+{C(X_0) }="base";
	(36,-20)*+{ C(X_1) \otimes M_{n_1} }="0";
	(36,20)*+{ C(X_1) \otimes M_{n_1} }="1";
	(72,-30)*+{ C(X_2) \otimes M_{n_2} }="00";
	(72,-10)*+{ C(X_2) \otimes M_{n_2} }="01";
	(72,30)*+{ C(X_2) \otimes M_{n_2} }="10";
	(72,10)*+{ C(X_2) \otimes M_{n_2} }="11";
	(108,-35)*+{ C(X_3) \otimes M_{n_3} }="000";
	(108,-25)*+{ C(X_3) \otimes M_{n_3} }="001";
	(108,-15)*+{ C(X_3) \otimes M_{n_3} }="010";
	(108,-5)*+{ C(X_3) \otimes M_{n_3} }="011";
	(108,35)*+{ C(X_3) \otimes M_{n_3} }="100";
	(108,25)*+{ C(X_3) \otimes M_{n_3} }="101";
	(108,15)*+{ C(X_3) \otimes M_{n_3} }="110";
	(108,5)*+{ C(X_3) \otimes M_{n_3} }="111";
	{\ar@2 "base";"0"};
	{\ar@2 "base";"1"};
	{\ar@2 "0";"01"};
	{\ar@2 "0";"00"};
	{\ar@2 "1";"11"};
	{\ar@2 "1";"10"};
	{\ar@2 "00";"000"};
	{\ar@2 "00";"001"};
	{\ar@2 "01";"010"};
	{\ar@2 "01";"011"};
	{\ar@2 "10";"100"};
	{\ar@2 "10";"101"};
	{\ar@2 "11";"110"};
	{\ar@2 "11";"111"};
\end{xy}
\]
The resulting inductive limit is simply the tensor product of the one we had 
before with the continuous functions on the Cantor set. The odometer action on 
the Cantor set provides an almost periodic action with the Rokhlin property on 
the inductive limit. To get a simple $C^*$-algebra, we add a small number of 
point evaluations which go across and connect the different summands; a similar 
idea was used to construct order 2 symmetries of the Elliott invariant in 
\cite{hirshberg-phillips-asymmetry} and later of the algebra itself in 
\cite{AGP}. (Those two papers were not published in chronological order.)
 The dotted lines in the 
diagram below indicate a small number of point evaluation maps.
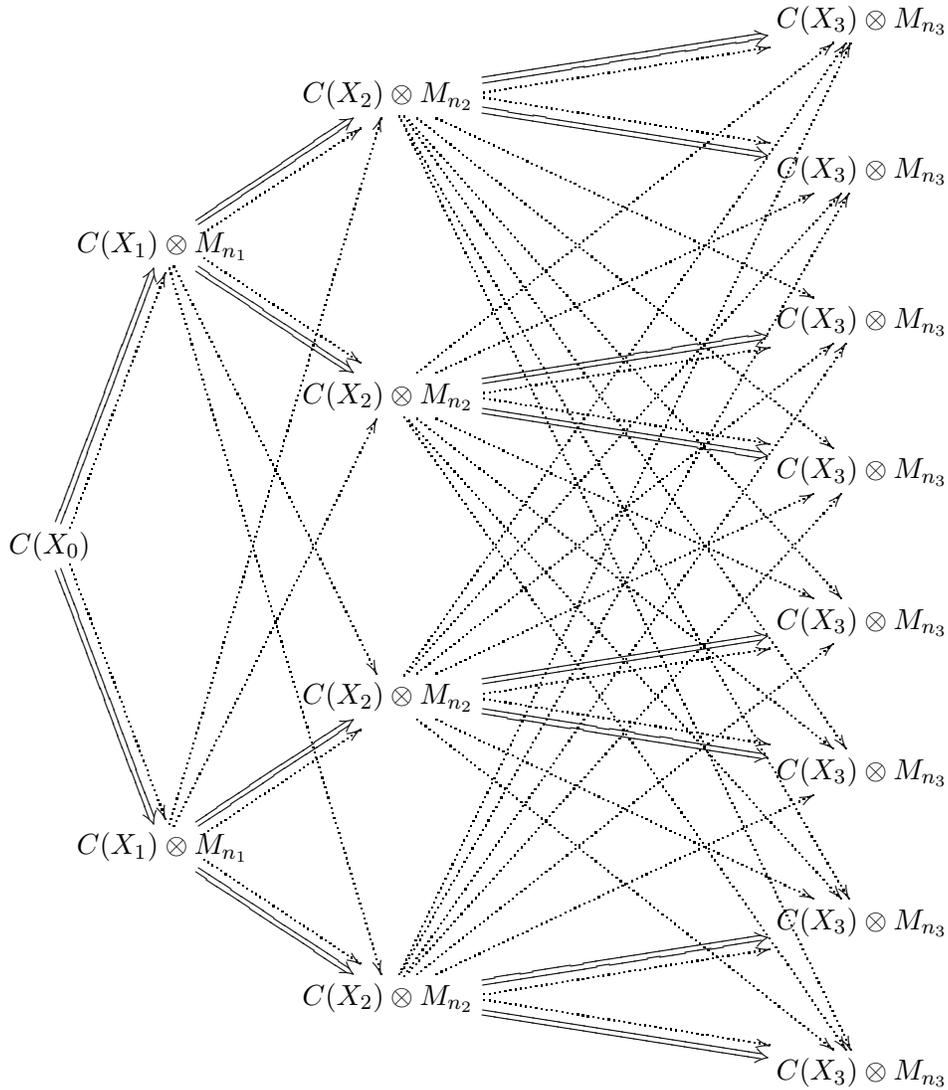
\[
\begin{xy}
	(0,0)*+{C(X_0) }="base";
	(15,-40)*+{ C(X_1) \otimes M_{n_1} }="0";
	(15,40)*+{ C(X_1) \otimes M_{n_1} }="1";
	(45,-60)*+{ C(X_2) \otimes M_{n_2} }="00";
	(45,-20)*+{ C(X_2) \otimes M_{n_2} }="01";
	(45,60)*+{ C(X_2) \otimes M_{n_2} }="10";
	(45,20)*+{ C(X_2) \otimes M_{n_2} }="11";
	(108,-70)*+{ C(X_3) \otimes M_{n_3} }="000";
	(108,-50)*+{ C(X_3) \otimes M_{n_3} }="001";
	(108,-30)*+{ C(X_3) \otimes M_{n_3} }="010";
	(108,-10)*+{ C(X_3) \otimes M_{n_3} }="011";
	(108,70)*+{ C(X_3) \otimes M_{n_3} }="100";
	(108,50)*+{ C(X_3) \otimes M_{n_3} }="101";
	(108,30)*+{ C(X_3) \otimes M_{n_3} }="110";
	(108,10)*+{ C(X_3) \otimes M_{n_3} }="111";
	{\ar@2 "base";"0"};
	{\ar@{.>}@<1ex> "base";"0"};
	{\ar@2 "base";"1"};
	{\ar@{.>}@<-1ex> "base";"1"};
	{\ar@2 "0";"01"};
	{\ar@{.>}@<-1ex> "0";"01"};
	{\ar@2 "0";"00"};
	{\ar@{.>}@<1ex> "0";"00"};
	{\ar@2 "1";"11"};
	{\ar@{.>}@<1ex> "1";"11"};
	{\ar@2 "1";"10"};
	{\ar@{.>}@<-1ex> "1";"10"};
	{\ar@{.>}
		"0";"10"};
	{\ar@{.>} "0";"11"};
	{\ar@{.>}
		"1";"01"};
	{\ar@{.>}
		"1";"00"};
	{\ar@2 "00";"000"};
	{\ar@{.>}@<1ex> "00";"000"};
	{\ar@2 "00";"001"};
	{\ar@{.>}@<-1ex> "00";"001"};
	{\ar@2 "01";"010"};
	{\ar@{.>}@<1ex> "01";"010"};
	{\ar@2 "01";"011"};
	{\ar@{.>}@<-1ex> "01";"011"};
	{\ar@2 "10";"100"};
	{\ar@{.>}@<-1ex> "10";"100"};
	{\ar@2 "10";"101"};
	{\ar@{.>}@<1ex> "10";"101"};
	{\ar@2 "11";"110"};
	{\ar@{.>}@<-1ex> "11";"110"};
	{\ar@2 "11";"111"};
	{\ar@{.>}@<1ex> "11";"111"};
	{\ar@{.>} "00";"010"};
	{\ar@{.>}
		"00";"011"};	
	{\ar@{.>}
		"00";"111"};
	{\ar@{.>}
		"00";"110"};
	{\ar@{.>}
		"00";"101"};
	{\ar@{.>}
		"00";"100"};
	{\ar@{.>}	"01";"000"};
	{\ar@{.>}	"01";"001"};
	{\ar@{.>}	"01";"100"};
	{\ar@{.>}	"01";"101"};
	{\ar@{.>}	"01";"110"};
	{\ar@{.>}	"01";"111"};
	{\ar@{.>}	"10";"000"};
	{\ar@{.>}	"10";"001"};
	{\ar@{.>}	"10";"010"};
	{\ar@{.>}	"10";"011"};
	{\ar@{.>}	"10";"110"};
	{\ar@{.>}	"10";"111"};
	{\ar@{.>}	"11";"000"};
	{\ar@{.>}	"11";"001"};
	{\ar@{.>}	"11";"100"};
	{\ar@{.>}	"11";"101"};
	{\ar@{.>}	"11";"010"};
	{\ar@{.>}	"11";"011"};
\end{xy}
\]
This will now provide a simple $C^*$-algebra, and with some fiddling around, we 
can modify the odometer action on the Cantor set to get an automorphism of this 
Villadsen-type algebra with similar properties; those properties were shown in 
\cite{hirshberg-JOT} to be sufficient to deduce that the crossed product is 
$\Zh$-stable. 

We now proceed to make all of this precise. We use similar notation to that 
used in \cite{hirshberg-phillips-asymmetry}, so that we can refer to technical 
lemmas already proved there, without the need to repeat them here.
We fix the following notation for our construction. The construction requires choosing a sequence of natural numbers $d(n)$ satisfying conditions (\ref{Cn_6918_General_dn}), (\ref{Cn_6918_General_dnkn}) and (\ref{Cn_6918_General_Size}) below. Any sequence satisfying those conditions would do; for instance, we could choose $d(n) = 10^n$. 
	\begin{enumerate}
		\item\label{Cn_6918_General_dn}
		$(d (n) )_{n = 0, 1, 2, \ldots}$
		is a sequence
		with $d (0) = 1$. 
		Moreover, we set $l(0) = 1$ and 
		for $n \in \N$, we set
		\[
		l (n) = d (n) + 2^{n-1} \, ,
		\qquad
		r (n) = \prod_{j = 0}^n l (j) \, ,
		\andeqn
		s (n) = \prod_{j = 0}^n d (j) \, .
		\]
		\item\label{Cn_6918_General_dnkn}
		We assume that $d(n) > 2^{n-1} $
		for all $n \in \N$. (In fact, this sequence will increase much faster 
		than $2^{n-1}$.)
		\item\label{Cn_6918_General_Size}
		We define
		\[
		\kp = \inf_{n \in \N} \frac{s (n)}{r (n)} \, .
		\]
		We assume that the sequences above are chosen so that $\kp > 
		\frac{1}{2}$.
		%
		\item\label{Cn_6918_General_Spaces} For $n = 0, 1, 2, \ldots$, we set 
		$(X_n) = (S^2)^{s(n)}$. 	We identify $X_{n+1} = X_n^{d(n+1)}$. 
		For $n \geq 0$ and $j = 1, 2, \ldots, d (n + 1)$,
		we let $P^{(n)}_j \colon X_{n + 1} \to X_{n}$
		be the $j$-th coordinate projection.
		\item\label{condition-points-dense}
		We choose points $x_m \in X_m$
		for $m \geq 0$ such that 
		for all $n \geq 0$,
		the set
		\begin{align*}
			& \big\{ \big( P^{(n)}_{\nu_{1}} \circ P^{(n + 1)}_{\nu_{2}}
			\circ \cdots \circ P^{(m - 1)}_{\nu_{m - n}} \big) (x_m) \mid
			\\
			& \hspace*{1em} {\mbox{}}
			{\mbox{$m > n$
					and $\nu_j \in  \{ 1, 2, \ldots, d (n + j) \}$
					for $j = 1, 2, \ldots, m - n$}} \big\}
		\end{align*}
		is dense in $X_n$.
		\item\label{Cn_6918_General_Algs}
		For $n = 0,1,2,\ldots$, we set
		\[
		C_n \cong C(X_n \times \Z_{2^n} , M_{r (n)} ) \; .
		\]
		We freely identify $C_n$ with 
		\[
		C_n \cong  M_{r (n)} \otimes C(X_n \times \Z_{2^n} ) \; .
		\]
		We also identify
		\[
		 M_{r (n)} \cong  M_{l (1)} \otimes  M_{l (2)} \otimes \ldots  M_{l (n)}
			\, .
		\]
		\item We let $\pi_{n,n+1} \colon \Z_{2^{n+1}} \to \Z_{2^n}$ denote 
		the map $\pi_{n,n+1} (k) = k ( \mathrm{mod } \;2^n )$.
		\item\label{Cn_6918_General_Maps}
		For $n \geq 0$,
		we define \cfn{s}
		\[
		S_{n, 1}, \, S_{n, 2}, \, \ldots, \, S_{n, \, l (n + 1)}
		\colon  (X_{n + 1} \times \Z_{2^{n+1}} )  \to X_n \times \Z_{2^n} 
		\]
		as follows:
		\begin{enumerate}
			\item\label{Cn_6918_General_1a_Maps_XLow}
			For $k \in \Z_{2^{n+1}}$, for $x \in X_{n + 1}$ and for $j = 1, 
			2, \ldots, d (n + 1)$, set
			\[
			S_{n, j} (x,k) = ( P^{(n)}_j (x) , \pi_{n,n+1}(k) )
			\, .
			\]
			\item\label{Cn_6918_General_1a_Maps_YTop}
			For
			$x \in X_{n + 1}$, $k \in \Z_{2^{n+1}}$ and $j = 0, 1, 2 , \ldots,  
			2^{n}-1$, set
			\[
			S_{n, d (n + 1) + 1 + j} (x,k) = (x_n, j )
			\]
			(In the last row above, we identify $	j = 0, 1, 2 
			, \ldots,  2^{n}-1 $ with the elements of $\Z_{2^n}$.)
		\end{enumerate}
		Next, 		
		we define unital \hm s
		\[
		\gm_n \colon
		C (X_n \times \Z_{2^n} ) 
		\to M_{l (n + 1)} \big( C (X_{n + 1} \times \Z_{2^{n+1}} ) \big) 
		\]
		by
			\[
		\gm_n (f)
		= \diag \big( f \circ S_{n, 1}, \, f \circ S_{n, 2},
		\, \ldots, \, f \circ S_{n, \, l (n + 1)} \big) \, .
		\]
		We then define
		\[
		\Gm_{n + 1, \, n} \colon C_n \to C_{n + 1}
		\]
		by
		$\Gm_{n + 1, \, n} = \id_{M_{r (n)}} \otimes \gm_n$.
		Moreover,
		for $ n \geq m \geq 0 $, we set
		\[
		\Gm_{n, m}
		= \Gm_{n, n - 1} \circ \Gm_{n - 1, \, n - 2} \circ \cdots
		\circ \Gm_{m + 1, m}
		\colon C_m \to C_n \, .
		\]
		In particular, $\Gm_{n, n} = \id_{C_n}$.
		\item\label{Cn_6918_General_Limit}
		We set $C = \dirlim_n C_n$,
		taken with respect to the maps $\Gm_{n, m}$.
		The maps associated with the direct limit
		will be called $\Gm_{\infty, m} \colon C_m \to C$
		for $m \geq 0$.
		\setcounter{CnsEnumi}{\value{enumi}}
		%

		\item \label{Cond_automorphisms}
		We define automorphisms $\alpha_n \colon C_n \to C_n$ as follows.
		Let $v_n \in  M_{l ( n )}$ be the unitary given as a direct sum of the 
		identity matrix of size $d(n)$ and a cyclic permutation of the last 
		$2^{n-1}$ elements, that is, 
		\[
		v_n = \sum_{j=1}^{d(n)}e_{j,j} +\left (  \sum_{j=1}^{2^{n-1}-1} 
		e_{d(n)+j,d(n)+j+1}  \right )
		+ e_{d(n)+2^{n-1},d(n)+1} \, .
		\]
		Define $u_n \in  M_{r ( n )} \cong M_{l(1)} \otimes M_{l(2)} \otimes 
		\cdots \otimes M_{l(n)}$ by
		\[
		u_n = v_1 \otimes v_2 \otimes \cdots \otimes v_n \, .
		\]
		Now define
		\[
		\alpha_n \colon  C_n \to C_n
		\]
		by 
		\[
		\alpha_n (f) (x,k) = u_n f(x,k+1_{\Z_{2^n}}) u_n^* \, .
		\]
	\end{enumerate}
Notice that for any $n \in \N$ we have 
\[
\frac{s (n+1)}{r (n+1)}  = \frac{s (n)}{r (n)} \cdot \frac{ d(n+1) }{ l(n+1)} = \frac{s (n)}{r (n)} \cdot \frac{ d(n+1) }{d(n+1) + 2^{n}} <  \frac{s (n)}{r (n)} \, ,
\]
 so the sequence $\left ( \frac{s (n)}{r (n)} \right )_{n = 1, 2, \ldots}$
is strictly decreasing.

By the choice of point evaluations in condition (\ref{condition-points-dense}) above, it follows from \cite[Proposition 
2.1]{DGNP} that the $C^*$-algebra $C$ is simple.

We claim now that for all $n \geq 0$, we have 
\[ 
\alpha_{n+1} \circ \Gm_{n + 1, 
n} = \Gm_{n + 1, n} \circ \alpha_n
\, .
\]
This claim implies that this consistent sequence of automorphisms defines an 
inductive 
automorphism $\alpha \colon C \to C$.
To prove the claim, fix $f \in C_n$ and a point $(x,k) 
\in X_{n+1} \times \Z_{2^{n+1}}$. We may assume that $f$ is of the form $f = a 
\otimes g$ for $a \in M_{r(n)}$ and $g \in C(X_n \times \Z_{2^n})$, as those 
span the entire space. Denote 
$\tilde{g}(x,k) = g(x,k+1_{\Z_{2^n}})$. 
We have
\begin{align*}
	\alpha_{n+1} &  \circ \Gm_{n + 1, n} (f) (x , k)  & \\
	& \stackrel{(\ref{Cond_automorphisms})}{=} u_{n+1} \Gm_{n + 1, n} (f) (x,k+1_{\Z_{2^{n+1}}} ) u_{n+1}^* & \\
	& \stackrel{(\ref{Cond_automorphisms})}{=}  u_n a u_n^* \otimes v_{n+1} \gamma_{n + 1, n} (g) 
	(x,k+1_{\Z_{2^{n+1}}}) v_{n+1}^* & \\
	& \stackrel{(\ref{Cn_6918_General_Maps})}{=}  u_n a u_n^* \otimes 	 v_{n+1}
	\diag \big( g (S_{n, 1} ( (x,k+1_{\Z_{2^{n+1}}}) ) ), 
	\ldots, \,
		& \\
		& \quad \quad	
	  g (S_{n, l (n + 1) } ( (x,k+1_{\Z_{2^{n+1}}}) ) )  \big)  
	 v_{n+1}^* & \\
	& \stackrel{(\ref{Cn_6918_General_Maps})}{=} u_n a u_n^* \otimes 	 v_{n+1}
	\diag \big( g (P^{(n)}_1 (x) , \pi_{n,n+1}(k+1_{\Z_{2^{n+1}}}) ), 
	\ldots, 
	& \\
	& \quad \quad g( P^{(n)}_{d(n+1)} (x) , \pi_{n,n+1}(k+1_{\Z_{2^{n+1}}}) ),  
	g (  (x_n, 
	0 ) ) , g (  (x_n, 1 ) ) 
	, \ldots, & \\
	&  	\quad \quad
	g (  (x_n, 2^{n}-1 ) ) \big)  v_{n+1}^*  & \\	
	&  \stackrel{(\ref{Cond_automorphisms})}{=} u_n a u_n^* \otimes 	
	\diag \big( g (P^{(n)}_1 (x) , \pi_{n,n+1}(k+1_{\Z_{2^{n+1}}}) ), 
	\ldots, 
	& \\
	& \quad \quad g( P^{(n)}_{d(n+1)} (x) , \pi_{n,n+1}(k+1_{\Z_{2^{n+1}}}) ),  
	g (  (x_n, 
	1 ) ) , g (  (x_n, 2 ) ) 
	, \ldots, & \\
	&  	\quad \quad
	g (  (x_n, 2^{n}-1 ) ) , g (  (x_n, 0 ) ) \big)    & \\	
	& = u_n a u_n^* \otimes 	
	\diag \big( \tilde{g} (P^{(n)}_1 (x) , \pi_{n,n+1}(k) ), 
	\ldots, 
	& \\
	& \quad \quad \tilde{g}( P^{(n)}_{d(n+1)} (x) , \pi_{n,n+1}(k) ),  
	\tilde{g} (  (x_n, 
	0 ) ) , \tilde{g} (  (x_n, 1 ) ) 
	, \ldots, & \\
	&  	\quad \quad
	\tilde{g} (  (x_n, 2^{n}-1 ) ) ) \big)    & \\
	&  \stackrel{(\ref{Cn_6918_General_Maps})}{=}  \Gm_{n + 1, n} \circ \alpha_{n} (f) (x , k) &
\end{align*}
 To complete the proof of the theorem, we 
need to show the following.
\begin{Claim-nonum}[A] $C \rtimes_{\alpha} \Z$ is $\Zh$-stable.
\end{Claim-nonum}
\begin{Claim-nonum}[B] $C$ is not $\Zh$-stable.
\end{Claim-nonum}
For Claim (A), we show that $\alpha$ has the Rokhlin property, and is nearly 
approximately inner in the sense of \cite[Section 4]{hirshberg-JOT}. It then 
follows from \cite[Theorem 4.1]{hirshberg-JOT} that $C \rtimes_{\alpha} \Z$ is 
$\Zh$-stable, as required. To see that $\alpha$ is nearly approximately inner, 
we note that the restriction of $\alpha$ to each subalgebra $C_n$ is in fact 
periodic, with period $2^n$. This condition is stronger than near approximate 
innerness. (See \cite[Proposition 4.2]{hirshberg-JOT}.) To see that $\alpha$ 
has the Rokhin property, it suffices to show the following stronger property: 
 given a finite set $F \subset 
C$, $\ep>0$ and $N \in \N$, there exists $m \geq N$ and projections 
$p_0,\ldots,p_{m-1}$ such that 
\begin{enumerate}
	\item $\sum_{k=0}^{m-1} p_k = 1$.
	\item $\|\alpha(p_k) - p_{k+1} \| < \eps$ for all $j$, with addition taken 
	modulo $m$.
	\item $\| [ p_k , a ] \|<\ep$ for $k=1,2,\ldots m$ and for all $a \in F$. 
\end{enumerate}
That is, we show that the definition of the Rokhlin property holds with a 
single tower for arbitrarily large tower lengths; here we obtain lengths which 
are powers of $2$.
By perturbing $F$, we may assume without loss of generality that $F \subseteq 
C_n$ for some $n$. By enlarging $n$, we may assume that $2^n \geq N$. Recall from condition (\ref{Cn_6918_General_Algs}) 
that $C_n \cong C (X_n \times \Z_{2^n},M_{r(n)} )$. For $k \in \Z_{2^n}$, define 
$p_k$ by $p_k(x,j) = \delta_{k,j} 1_{M_{r(n)}}$. Then the projections 
$p_0,p_1,\ldots,p_{2^n-1}$ are central in $C_n$ (and in particular commute with 
$F$), add up to 1, and they satisfy $\alpha(p_k) = p_{k+1}$ for all $k$.

For Claim (B), we proceed as in \cite{hirshberg-phillips-asymmetry} to show 
that $C$ has positive radius of comparison, and in particular, it is not 
$\Zh$-stable, as it was shown in \cite{rordam-IJM} that $\Zh$-stable 
$C^*$-algebras have strict comparison. In fact, we claim that $\rc (C) 
\geq 2 \kappa - 1$. Notice that by condition (\ref{Cn_6918_General_Size}), we have $2\kappa - 1 > 0$, so this will show that $\rc(C) > 0$, and in particular that $C$ is not $\Zh$-stable. 

Indeed, suppose $\rh < 2 \kappa - 1$.
We show that $C$ does not have $\rh$-comparison.
We shall use projections to witness the failure of comparison. For projections, 
a trace $\tau$ coincides with the associated dimension function $d_{\tau}$.

Let $L$ be the tautological line bundle over
$S^2 \cong \C \mathbb{P}^1$, and let $b \in C (S^2, M_2)$ denote the Bott projection, that is, the projection onto the section space of~$L$.
We recall (\cite[Lemma 1.9]{hirshberg-phillips-asymmetry}) that the Cartesian 
product $L^{\times k}$
does not embed in a trivial bundle over $(S^2)^k$
of rank less than $2k$.
For $n > 0$, set $b_n = (\id_{M_2} \otimes \tgamma_{n, 0}) (b)
\in M_2 (C_n)$. On each of the connected components of $X_n \times \Z_{2^n}$, the projection $b_n$ is a direct sum of a projection of rank $s(n)$ onto the section space of the Cartesian product $L^{\times s(n)}$ and a trivial bundle of rank $r(n) - s(n)$. Note that by condition (\ref{Cn_6918_General_Size}), we have $s(n) > r(n) - s(n)$ for any $n \in \N$.

Choose $n \in \N$ such that $1 / {r (n)} < 2 \kp - 1 - \rh$.
Choose $M \in \N$ such that $\rh + 1 < M / {r (n)} < 2 \kp$.
Let $e \in M_{\infty} (C_n)$ be a trivial \pj{}
of rank~$M$.
By slight abuse of notation,
we use $\tgamma_{m, n}$ to denote the amplified map
from $M_{\infty} (C_n)$ to $M_{\infty} (C_m)$ as well.
For $m > n$, the rank of $\tgamma_{m, n} (e)$
is $M \cdot \frac{r (m)}{r (n)}$,
and the choice of $M$ guarantees that
this rank is strictly less than $2 s (m)$.
Now, for any trace $\tau$ on $C_m$
(and thus for any trace on $C$),
given that $b$ is a rank $1$ projection, we have
\[
\tau (\tgamma_{m, n} (e))
= \frac{1}{r (m)} \cdot M \cdot \frac{r (m)}{r (n)} = \frac{M}{r (n)} > 1+\rho 
= \tau(b_m) + \rho \, .
\]

On the other hand, if
$\tgamma_{\infty, 0} (b) \precsim \tgamma_{\infty, n} (e)$
then, in particular, there exists some $m > n$
and $x \in M_{\infty} (C_m)$
such that $\|x\tgamma_{m, n} (e)x^* - b_m\| < \frac{1}{2}$,
and thus $b_m$ is Murray von-Neumann equivalent to a subprojection of 
$\tgamma_{m, n} (e)$. In particular, it shows that the projection onto the 
section space of the Cartesian product $L^{\times s(m)}$ embeds into a trivial 
bundle of rank $M \cdot \frac{r (m)}{r (n)} < 2 s (m)$. This is a contradiction.

 (We can in fact compute the radius 
of comparison exactly, and obtain different radii of comparison by picking 
different choices of parameters, but this is not important here.) 
\end{proof}

We conclude with some questions.
\begin{qst}
	We don't have control over the trace simplex of the $C^*$-algebra $B$ in 
	Corollary~\ref{cor:main}. One expects that it might be the Poulsen simplex. 
	Is there an example of a $C^*$-algebra $B$ as in Corollary~\ref{cor:main} 
	which furthermore has unique trace?
\end{qst}
\begin{qst}
	Does there exist a simple unital separable nuclear $\Zh$-stable 
	$C^*$-algebra 
	$B$ along with a point-norm continuous action $\beta 
	\colon \R \to \Aut(B)$ such that $B \rtimes_{\beta} \R$ is simple but not 
	$\Zh$-stable?  In light of the results for the circle, it seems reasonable 
	to conjecture that there 
	should be such an example.
\end{qst}
\begin{qst}
	Let $B$ be a simple unital separable nuclear stably finite $\Zh$-stable 
	$C^*$-algebra. Suppose $\beta$ is an action of $\T$ or $\R$ on $B$ with 
	full strong Connes spectrum. Suppose that the crossed product is 
	$\Zh$-stable. Does the action necessarily have finite Rokhlin dimension? 
	That may be one possible generalization of the results for discrete group 
	actions mentioned in the introduction (possibly with some restrictions on 
	the trace simplex).
\end{qst}
\begin{qst}
	Let $A$ be a simple unital separable nuclear stably finite $C^*$-algebra with positive radius of comparison. Let $\alpha$ be a properly outer automorphism of $A$. What are the possible values of $\rc(A \rtimes_{\alpha} \Z)$? How are those related to $\rc(A)$? (See \cite{AGP,asadi-vasfi} for some partial related results concerning actions of finite groups.)
\end{qst}
\begin{rmk}
	The choice of powers of $2$ was arbitrary. We could have easily made 
	choices to `overlay' such a $C^*$-dynamical system over any other odometer 
	action. Furthermore, by replacing spheres by contractible spaces, as was 
	done in \cite{toms-counterexample,hirshberg-phillips-asymmetry}, we could 
	have obtained examples with more tractable $K$-theory. 
	
	We remark that the Bratteli diagram we use for the construction of those 
	odometer actions are not the ordered Bratteli-Vershik diagrams used to 
	model more general Cantor dynamical systems, as done in \cite{HPS}. 
\end{rmk}


\end{document}